\documentclass{elsarticle}

\usepackage{amsmath}
\usepackage{amsfonts}
\usepackage{amssymb}
\usepackage{amsthm}
\usepackage{graphicx}
\usepackage{enumitem}
\usepackage{hyperref}
\usepackage{microtype}

\newcommand{\D}{\mathcal{D}}
\newcommand{\A}{\mathcal{A}}
\newcommand{\e}{\epsilon}
\newcommand{\de}{\delta}
\newcommand{\ff}{\varphi}
\newcommand{\spn}{\operatorname{span}}
\newcommand{\la}{\lambda}
\def \<{\left\langle}
\def \>{\right\rangle}

\newtheorem{Theorem}{Theorem}[section]
\newtheorem{Lemma}{Lemma}[section]

\newtheorem{Proposition}{Proposition}[section]
\newtheorem*{Remark}{Remark}

\newtheorem{Corollary}{Corollary}[section]
\numberwithin{equation}{section}

\journal{Applied and Computational Harmonic Analysis}

\begin{document}

\title{Biorthogonal Greedy Algorithms in Convex Optimization}

\author[1]{A.\,V.~Dereventsov\corref{cor}}
\ead{dereventsov@gmail.com}
\author[2]{V.\,N.~Temlyakov}
\ead{temlyakovv@gmail.com}

\cortext[cor]{Corresponding author}
\address[1]{Oak Ridge National Laboratory}
\address[2]{University of South Carolina, Steklov Institute of Mathematics, and Lomonosov Moscow State University}

\date{}

\begin{abstract}
The study of greedy approximation in the context of convex optimization is becoming a promising research direction as greedy algorithms are actively being employed to construct sparse minimizers for convex functions with respect to given sets of elements.
In this paper we propose a unified way of analyzing a certain kind of greedy-type algorithms for the minimization of convex functions on Banach spaces.
Specifically, we define the class of Weak Biorthogonal Greedy Algorithms for convex optimization that contains a wide range of greedy algorithms.
We analyze the introduced class of algorithms and establish the properties of convergence, rate of convergence, and numerical stability, which is understood in the sense that the steps of the algorithm are allowed to be performed not precisely but with controlled computational inaccuracies.
We show that the following well-known algorithms for convex optimization~--- the Weak Chebyshev Greedy Algorithm (co) and the Weak Greedy Algorithm with Free Relaxation (co)~--- belong to this class, and introduce a new algorithm~--- the Rescaled Weak Relaxed Greedy Algorithm (co).
Presented numerical experiments demonstrate the practical performance of the aforementioned greedy algorithms in the setting of convex minimization as compared to optimization with regularization, which is the conventional approach of constructing sparse minimizers.
\end{abstract}


\maketitle

\section{Introduction}\label{sec:intro}
This paper is devoted to the question of building an approximate sparse solution of a given convex optimization problem.
A typical setting is to find an approximate minimum of a real-valued convex function $E$ defined on the Banach space $X$, i.e.
\begin{equation}\label{eq:opt}
	\text{find}\ \ x^* = \mathop{\operatorname{argmin}}_{x \in X} E(x).
\end{equation}
When optimization is performed over the whole space $X$, it is called an {\it unconstrained optimization problem}.
Usually in practice it is desirable to obtain a minimizer that possesses a certain structure or belongs to a given domain $S \subset X$, in which case problem~\eqref{eq:opt} becomes a {\it constrained optimization problem}.

In particular, it is often preferable that the constructed solution $x^*$ is sparse with respect to a given set of elements $\D \subset X$.
For instance, such a problem is natural in the context of reduced order modeling, see e.g.~\cite{binev2011convergence}, \cite{buffa2012priori}, \cite{dereventsov2019natural}, \cite{rozza2007reduced}, where it is required to construct a set of elements $\mathcal{V}$, called a reduced basis, that approximates the high-dimensional manifold $\mathcal{D} \subset X$ of interest with respect to a given metric $E : X \to \mathbb{R}$, which is often a norm of $X$ but may be changed to fit the particular application.
One approach to construct a reduced basis $\mathcal{V}$ is to select a set of elements $\{g_1, \ldots, g_n\} \in \mathcal{F}$ such that any other element $f \in \mathcal{F}$ can be sufficiently well approximated by the set $\{g_1, \ldots, g_n\}$.
Since the quality of the approximation is measured by the given metric $E$, one essentially has to solve the sparse minimization problem $\min E(f - \sum_{g\in\mathcal{F}} c_g g)$ for each $f \in \mathcal{F}$ and then select $\mathcal{V} := \{g_1, \ldots, g_n\}$ as the most frequent members of the solutions.

A conventional approach to solve a task of obtaining an approximate sparse minimizer is to impose an additional $\ell_1$-regularization on the original problem (see e.g.~\cite{FNW}) and, instead of~\eqref{eq:opt}, solve the problem
\begin{equation}\label{eq:opt_reg}
	\text{find}\ \ x^* = \mathop{\operatorname{argmin}}_{x \in X} \big( E(x) + \lambda \|x\|_\D \big),
\end{equation}
where $\lambda > 0$ is an appropriate regularization parameter and $\|\cdot\|_\D$ is the atomic norm with respect to the set $\D$ (see e.g.~\cite{chandrasekaran2012convex}), i.e.
\begin{equation}\label{eq:D_norm}
	\|x\|_\D := \inf \left\{ \sum_{g \in \D} |c_g| : x = \sum_{g \in \D} c_g \, g \right\}.
\end{equation}
Even though such an approach is quite popular, it might not always result in the most appropriate solution since, as with any regularization, it essentially changes the target function in order to promote sparsity of the solution.
While such a formulation is equivalent to the constrained minimization problem, the restriction on the atomic norm acts as only a proxy for the solution sparsity and therefore does not necessarily guarantee the desired outcome.
Moreover, one has to select the value of the introduced regularization coefficient $\lambda$, which in practice is typically performed heuristically and often has to be tweaked to achieve a satisfactory performance.

Another way of obtaining a sparse minimizer (without changing the optimization problem) is to procedurally construct a sequence of minimizers with an increasing support, or, more generally, to design an algorithm that after $m$ iterations provides a point $x_m$ such that $E(x_m)$ is close to the $\inf_{x \in S} E(x)$ and that $x_m$ is $m$-sparse with respect to $\D$, i.e.
\[
	x_m = \sum_{j=1}^m c_j \, g_j
	\ \ \text{with}\ \ 
	g_1, \ldots, g_m \in \D
	\ \ \text{and}\ \ 
	c_1, \ldots, c_m \in \mathbb{R}.
\]

A wide class of algorithms that fit such requirements is the class of greedy algorithms in approximation theory, see e.g.~\cite{D}, \cite{VTbook}.
A typical problem of greedy approximation is the following.
Let $X$ be a Banach space with the norm $\|\cdot\|$ and let $\D$ be a dictionary, i.e. a dense set of semi-normalized elements of $X$.
The goal of a greedy algorithm is, for a given element $f \in X$, to obtain a sparse approximation with respect to the dictionary $\D$.
Greedy algorithms are iterative by design and generally after $m$ iterations a greedy algorithm provides an $m$-term linear combination of the elements from the dictionary $\D$ that approximates the given element $f$.

It is easy to reframe a greedy approximation problem as a problem of convex function optimization.
Indeed, for a given dictionary $\D$ consider the set of all $m$-term linear combinations with respect to $\D$ ($m$-sparse with respect to $\D$ elements):
\[
	\Sigma_m(\D) := \left\{ x \in X: x = \sum_{j=1}^m c_j \, g_j, \ \ g_1, \ldots, g_m \in \D \right\}.
\]
Greedy algorithms in approximation theory are designed to provide a simple way to build good approximants of $f$ from $\Sigma_m(\D)$, hence the problem of greedy approximation is the following:
\begin{equation}\label{eq:opt_ga}
	\text{find}\ \ x_m = \mathop{\operatorname{argmin}}_{x \in \Sigma_m} \|f - x\|.
\end{equation}
Clearly, problem~\eqref{eq:opt_ga} is a constrained optimization problem of the real-valued convex function $E(x) := \|f - x\|$ over the manifold $\Sigma_m(\D) \subset X$.

At first glance the settings of approximation and optimization problems appear to be very different since in approximation theory our task is to find a sparse approximation of a given element $f \in X$, while in optimization theory we want to find an approximate sparse minimizer of a given target function $E : X \to \mathbb{R}$ (for instance, energy function or loss function).
However it is now well understood that similar techniques can be used for solving both problems.
Namely, it was shown in~\cite{VT140} and in follow up papers (see, for instance, \cite{DT}, \cite{GP}, \cite{NP}, \cite{VT141}, and~\cite{VT148}) how methods developed in nonlinear approximation theory (greedy approximation techniques in particular) can be adjusted to find an approximate sparse (with respect to a given dictionary $\D$) solution to the optimization problem~\eqref{eq:opt}.
Moreover, there is an increasing interest in building such sparse approximate solutions using different greedy-type algorithms, for example, \cite{BD1}, \cite{BD2}, \cite{chandrasekaran2012convex}, \cite{Cl}, \cite{Ja2}, \cite{JS}, \cite{SSZ}, \cite{TRD}, and~\cite{Z}.

With an established framework it is straightforward to adjust a greedy strategy to a context of convex optimization; however, each of these modified techniques requires an individual analysis to guarantee a desirable performance.
On the other hand, it is known that the behavior of a greedy method is largely determined by the underlying geometry of the problem setting.
In particular, in~\cite{VT165} we present a unified way of analyzing different greedy-type algorithms in Banach spaces.
Specifically, we define the class of Weak Biorthogonal Greedy Algorithms ($\mathcal{WBGA}$) and prove convergence and rate of convergence results for algorithms from this class.
Such an approach allows for a simultaneous analysis of a wide range of seemingly different greedy algorithms based on the smoothness characteristic of the problem.

In this paper we adopt the approach of unified analysis for the setting of convex minimization.
In Section~\ref{sec:wbga} we adjust the class $\mathcal{WBGA}$ of algorithms designed for greedy approximation in Banach spaces and derive the class of Weak Biorthogonal Greedy Algorithms for convex optimization ($\mathcal{WBGA}$(co)), which consists of greedy algorithms designed for convex optimization.
We prove convergence and rate of convergence results for algorithms from the class $\mathcal{WBGA}$(co) in Theorems~\ref{thm:wbga_conv} and~\ref{thm:wbga_rate} respectively.
Thus, results in Section~\ref{sec:wbga} address two important characteristics of an algorithm~--- convergence and rate of convergence.

The rate of convergence is an essential characteristic of an algorithm, though in certain practical applications resistance to various perturbations might be of equal importance.
A systematic study of the stability of greedy algorithms in Banach spaces was started in~\cite{T7} and further advanced in~\cite{De}, where necessary and sufficient conditions for the convergence of a certain algorithm were obtained.
A transition to the optimization setting was performed in~\cite{DT} and~\cite{VT148}, where stability results for greedy-type algorithms for convex optimization were obtained.
In Section~\ref{sec:awbga} we discuss the stability of the algorithms from the class $\mathcal{WBGA}$(co) by analyzing the convergence properties of the algorithms from $\mathcal{WBGA}$(co) under the assumption of imprecise calculations in the steps of the algorithms.
We call such algorithms {\it approximate greedy algorithms} or {\it algorithms with errors}.
We prove convergence and rate of convergence results for the Weak Biorthogonal Greedy Algorithms with errors, which describes the stability of the algorithms from the class $\mathcal{WBGA}$(co)~--- an important characteristic that is crucial for practical implementation.

Since theoretical analysis cannot always predict the practical behavior of an algorithm, it is of interest to observe its actual implementation for particular problems.
In Section~\ref{sec:numerics} we demonstrate the performance of some algorithms from the class $\mathcal{WBGA}$(co) by employing them to solve various minimization problems.
Additionally, we compare these algorithms with a conventional method of obtaining sparse minimizers~--- optimization with $\ell_1$-regularization~\eqref{eq:opt_reg}.
Lastly, in Sections~\ref{sec:proofs_wbga} and~\ref{sec:proofs_awbga} we prove the results stated in Sections~\ref{sec:wbga} and~\ref{sec:awbga} respectively.

\section{Weak Biorthogonal Greedy Algorithms for Convex Optimization}\label{sec:wbga}
In this section we introduce and discuss the class of Weak Biorthogonal Greedy Algorithms for optimization of convex objective functions, denoted as $\mathcal{WBGA}$(co).
We begin by recalling the relevant terminology.

\subsection{Preliminaries}
Let $X$ be a real Banach space with the norm $\|\cdot\|$.
We say that a set of elements $\D$ from $X$ is a dictionary if each $g \in \D$ has the norm bounded by one and $\D$ is dense in $X$, that is
\[
	\|g\| \le 1
	\ \ \text{for any}\ \ 
	g \in \D,
	\ \ \text{and}\ \
	\overline{\spn\D} = X.
\]
For notational convenience in this paper we consider {\it symmetric dictionaries}, i.e. such that
\[
	g\in \D \ \ \text{implies} \ \ -g \in \D.
\]
We denote the closure (in $X$) of the convex hull of $\D$ by $\A_1(\D)$:
\begin{equation}\label{eq:A_1(D)}
	\A_1(\D) := \overline{\mathrm{conv} \D},
\end{equation}
which is the standard notation in relevant greedy approximation literature. 

Throughout the paper we assume that the target function $E$ is Fr{\'e}chet-differentiable, i.e. that at any $x \in X$ there is a bounded linear functional $E'(x) : X \to \mathbb{R}$ such that
\[
	\sup_{\|y\|=1} \Big( \lim_{u \to 0} \frac{E(x + uy) - E(x)}{u} - \< E'(x), y \> \Big) = 0.
\]
Then the convexity of $E$ implies that for any $x,y \in D$
\begin{equation}\label{eq:E'_conv1}
	E(y) \ge E(x) + \<E'(x),y-x\>,
\end{equation}
or, equivalently,
\begin{equation}\label{eq:E'_conv2}
	E(x) - E(y) \le \<E'(x),x - y\> = \<-E'(x),y-x\>.
\end{equation}

\begin{Remark}
We note that the condition of Fr{\'e}chet-differentiability is not necessary for the formulation of greedy algorithms and can be relaxed by considering support functionals in place of the derivative of $E$, as is done in~\cite[Chapter~5]{dereventsov2017convergence}.
Although the existence of support functionals is guaranteed by the convexity of the target function, to the best of our knowledge, no nontrivial results regarding the rate of convergence of greedy algorithms can be stated, hence we additionally impose the assumption of Fr{\'e}chet-differentiability.
\end{Remark}

The modulus of smoothness $\rho(E,S,u)$ of a function $E : X \to \mathbb{R}$ on a set $S \subset X$ is defined as
\begin{equation}\label{eq:mod_smt}
	\rho(E,S,u) := \frac{1}{2} \sup_{x\in S, \|y\|=1} \Big| E(x + uy) + E(x - uy) - 2E(x) \Big|.
\end{equation}
We note that, in comparison to the modulus of smoothness of a norm (see, for instance,~\cite[Part~3]{beauzamy2011introduction}), the modulus of smoothness of a function additionally depends on the chosen set $S \subset X$.
That is because a norm is a positive homogeneous function, thus its smoothness on the whole space is determined by its smoothness on the unit sphere, which is not the case for a general function on a Banach space.

The function $E$ is uniformly smooth on $S \subset X$ if $\rho(E,S,u) = o(u)$ as $u \to 0$.
We say that the modulus of smoothness $\rho(E,S,u)$ is of power type $1 \leq q \leq 2$ if $\rho(E,S,u) \leq \gamma u^q$ for some $\gamma > 0$.
Note that the class of functions with the modulus of smoothness of a nontrivial power type is completely different from the class of uniformly smooth Banach spaces with the norms of a nontrivial power type since any uniformly smooth norm is not uniformly smooth as a function on any set containing $0$.
However, it is shown in~\cite{borwein2009uniformly} that if a norm $\|\cdot\|$ has the modulus of smoothness of power type $q \in [1,2]$, then the function $E(\cdot) := \|\cdot\|^q$ has the modulus of smoothness $\rho(E,S,u)$ of power type $q$ for any set $S \subset X$.
In particular, it implies (see e.g.~\cite[Lemma B.1]{donahue1997rates}) that for any $1 \le p < \infty$ the function $E : L_p \to \mathbb{R}$ defined as
\[
	E_p(x) = \|x\|_{L_p}^p
\]
has the modulus of smoothness that satisfies
\[
\rho_p(E,X,u) \le
	\left\{\begin{array}{ll}
		\frac{1}{p} u^p & 1 \le p \le 2,
		\\
		\frac{p-1}{2} u^2 & 2 \le p < \infty,
	\end{array}\right.
\]
i.e. $\rho_p(E,X,u)$ is of power type $\min\{p,2\}$.

A typical smoothness assumption in convex optimization is of the form
\[
	|E(x + uy) - E(x) - \<E'(x),uy\>| \le Cu^2
\]
with some constant $C > 0$ and any number $u \in \mathbb{R}$ and elements $x,y \in X$ with $\|y\| = 1$.
In terms of the modulus of smoothness~\eqref{eq:mod_smt} such an assumption corresponds to the case $\rho(E,X,u) \le Cu^2 / 2$, i.e. that the modulus of smoothness of $E$ is of power type $2$.

\smallskip\noindent
{\bf Assumptions.}
In this work we assume that the target function $E : X \to \mathbb{R}$ is convex and Fr{\'e}chet-differentiable, and that the set
\[
	D = D(E) := \big\{ x \in X : E(x) \le E(0) \big\} \subset X
\]
is bounded.
Note that in this case we have
\[
	\inf_{x \in X} E(x) = \inf_{x \in D} E(x) > -\infty
	\ \ \text{and}\ \
	\mathop{\operatorname{argmin}}_{x \in X} E(x) = \mathop{\operatorname{argmin}}_{x \in D} E(x) \in D,
\]
i.e. the minimizer of $E$ exists and is unique and attainable.

\subsection{Weak Biorthogonal Greedy Algorithms}
Typically in greedy approximation one has to perform a greedy selection from a given dictionary $\D$, which might not always be possible.
In order to guarantee the feasibility of algorithms, it is conventional to perform a {\it weak} greedy step where the greedy search is relaxed.
Such relaxations are represented by a given sequence $\tau := \{t_m\}_{m=1}^\infty$, referred to as a {\it weakness sequence}.

For a convex Fr{\'e}chet-differentiable target function $E : X \to \mathbb{R}$ we define the following class of greedy algorithms.
\\[.5em]\noindent
{\bf Weak Biorthogonal Greedy Algorithms ($\boldsymbol{\mathcal{WBGA}}$(co)).\\}
We say that an algorithm belongs to the class $\mathcal{WBGA}$(co) with a weakness sequence $\tau = \{t_m\}_{m=1}^\infty$, $t_m\in[0,1]$, if sequences of approximators $\{G_m\}_{m=0}^\infty$ and selected elements $\{\varphi_m\}_{m=1}^\infty$ of the dictionary $\D$ satisfy the following conditions at every iteration $m \ge 1$:
\begin{enumerate}[label=\bf(\arabic*), leftmargin=.5in]
	\item\label{wbga_gs}
		Greedy selection: ${\displaystyle \<-E'(G_{m-1}), \varphi_m\> \ge t_m \sup_{\varphi\in\D} \< -E'(G_{m-1}), \varphi \>}$;
	\item\label{wbga_er}
		Error reduction: ${\displaystyle E(G_m) \le \inf_{\la\ge0} E(G_{m-1} + \la\ff_m)}$;
	\item\label{wbga_bo}
		Biorthogonality: ${\displaystyle \<E'(G_{m}), G_m\> = 0}$.
\end{enumerate}

Due to the imposed Assumptions, there exists a nontrivial and attainable minimum of the objective function $E$.
Then by condition~\ref{wbga_er} the sequence of $m$-sparse approximants $\{G_m\}_{m=0}^\infty$ constructed by an algorithm from the $\mathcal{WBGA}$(co) satisfies the relation
\[
	E(0) = E(G_0) \ge E(G_1) \ge E(G_2) \ge \dots,
\]
which guarantees that $G_m \in D$ for all $m \ge 0$.

\begin{Remark}
In the case $E(x) := \|f - x\|^q$ with any $f\in X$ and $q \ge 1$, the class $\mathcal{WBGA}$(co) coincides with the class $\mathcal{WBGA}$ from the approximation theory, which is introduced and analyzed in~\cite{VT165}.
\end{Remark}

\subsection{Examples of algorithms from the $\mathcal{WBGA}$(co)}\label{sec:wbga_ga}
In this section we briefly overview a few particular algorithms from the class $\mathcal{WBGA}$(co) that will be utilized in the numerical experiments presented in Section~\ref{sec:numerics}.
By $\tau := \{t_m\}_{m=1}^\infty$ we denote a weakness sequence, i.e. a given sequence of non-negative numbers $t_m \le 1$, $m = 1,2,3,\dots$.

We first define the Weak Chebyshev Greedy Algorithm for optimization of convex objective functions that is introduced and studied in~\cite{VT140}.
\\[.5em]\noindent
{\bf Weak Chebyshev Greedy Algorithm (WCGA(co)).\\}
Set $G^c_0 = 0$ and for each $m \ge 1$ perform the following steps:
\begin{enumerate}
	\item Take any $\varphi^{c}_m \in \D$ satisfying
		\[
			\< -E'(G^c_{m-1}), \varphi^c_m \> \ge t_m \sup_{\varphi\in\D} \< -E'(G^c_{m-1}), \varphi \>;
		\]
	\item Denote $\Phi_m^c = \spn \{\varphi^c_k\}_{k=1}^m$ and find $G_m^c \in \Phi_m^c$ such that
		\[
			E(G_m^c) = \inf_{G \in \Phi_m^c} E(G).
		\]
\end{enumerate}

Another algorithm, which utilizes a simpler approach to updating the approximant is the Weak Greedy Algorithm with Free Relaxation for optimization of convex objective functions (see~\cite{VT140}).
\\[.5em]\noindent
{\bf Weak Greedy Algorithm with Free Relaxation (WGAFR(co)).\\}
Set $G^f_0 = 0$ and for each $m \ge 1$ perform the following steps:
\begin{enumerate}[label=\bf(\arabic*), leftmargin=.5in]
	\item Take any $\varphi^f_m \in \D$ satisfying
		\[
			\< -E'(G^f_{m-1}), \varphi^f_m \> \ge t_m \sup_{\varphi\in\D} \< -E'(G^f_{m-1}), \varphi \>;
		\]
	\item Find $\omega_m \in \mathbb{R}$ and $ \lambda_m \in \mathbb{R}$ such that
		\[
			E\big( (1-\omega_m) G^f_{m-1} + \lambda_m \varphi^f_m \big)
			= \inf_{ \la, \omega \in \mathbb{R}} E\big( (1 - \omega) G^f_{m-1} + \la \varphi^f_m \big)
		\]
		and define $G^f_m = (1 - \omega_m) G^f_{m-1} + \la_m \varphi^f_m$.
\end{enumerate}

The next algorithm~--- the Rescaled Weak Relaxed Greedy Algorithm for optimization of convex objective functions~--- is an adaptation of its counterpart from the approximation theory (see~\cite{VT165}) that can be viewed as a generalization of the Rescaled Pure Greedy Algorithm, introduced in~\cite{Pet} and adapted for convex optimization in~\cite{GP}.
\\[.5em]\noindent
{\bf Rescaled Weak Relaxed Greedy Algorithm (RWRGA(co)).\\}
Set $G^r_0 = 0$ and for each $m \ge 1$ perform the following steps:
\begin{enumerate}
	\item Take any $\varphi^r_m \in \D$ satisfying
		\[
			\< -E'(G^r_{m-1}), \varphi^r_m \> \ge t_m \sup_{\varphi\in\D} \< -E'(G^r_{m-1}), \varphi \>;
		\]
	\item Find $\lambda_m \ge 0$ such that
		\[
			E(G^r_{m-1} + \la_m \varphi^r_m) = \inf_{\la \ge 0} E(G^r_{m-1} + \la \varphi^r_m);
		\]
	\item Find $\mu_m \in \mathbb{R}$ such that
		\[
			E\big( \mu_m (G^r_{m-1} + \la_m \varphi^r_m) \big)
			= \inf_{\mu \in \mathbb{R}} E\big( \mu (G^r_{m-1} + \la_m \varphi^r_m) \big)
		\]
		and define $G^r_m = \mu_m (G^r_{m-1} + \la_m \varphi^r_m)$.
\end{enumerate}

\begin{Proposition}\label{prp:ga_wbga}
The WCGA(co), the WGAFR(co), and the RWRGA(co) belong to the class $\mathcal{WBGA}$(co).
\end{Proposition}

\subsection{Convergence results for the $\mathcal{WBGA}(co)$}
In this section we state the results related to convergence and the rate of convergence for algorithms from the class $\mathcal{WBGA}$(co).

Our setting of an infinite dimensional Banach space makes the formulation of convergence results nontrivial, and thus we require a special sequence which is defined for a given modulus of smoothness $\rho(u) := \rho(E,D,u)$ and a given weakness sequence $\tau = \{t_m\}_{m=1}^\infty$.

Let $E : X \to \mathbb{R}$ be a convex uniformly smooth function, then $\rho(u) := \rho(E,D,u) : \mathbb{R} \to \mathbb{R_+}$ is an even convex function.
Assume that $\rho(u)$ has the property $\rho(1/\theta_0) \ge 1$ for some $\theta_0 \in (0,1]$ and
\[
	\lim_{u\to 0} \rho(u)/u = 0.
\]
Note that assumptions on uniform smoothness of $E$ and boundedness of domain $D \subset X$ guarantee the above properties.
Then for a given $0 < \theta \le \theta_0$ define $\xi_m := \xi_m(\rho,\tau,\theta)$ as the solution of the equation
\begin{equation}\label{eq:theta}
	\rho(u) = \theta t_m u.
\end{equation}
Note that conditions on $\rho(u)$ imply that the function
\[
	s(u) := \left\{\begin{array}{ll}
		\rho(u)/u, & u \neq 0
		\\
		0, & u = 0
	\end{array}\right.
\]
is continuous and increasing on $[0,\infty)$ with $s(1/\theta_0) \ge \theta_0$.
Thus equation~\eqref{eq:theta} has the unique solution $\xi_m = s^{-1}(\theta t_m)$ such that $0 < \xi_m \le 1/\theta_0$.

We now formulate our main convergence result for the $\mathcal{WBGA}$(co).
\begin{Theorem}\label{thm:wbga_conv}
Let $E$ be a uniformly smooth on $D \subset X$ convex function with the modulus of smoothness $\rho(E,D,u)$.
Assume that a sequence $\tau := \{t_m\}_{m=1}^\infty$ satisfies the condition that for any $\theta \in(0,\theta_0]$ we have
\[
	\sum_{m=1}^\infty t_m \xi_m(\rho,\tau,\theta) = \infty.
\]
Then for any algorithm from the class $\mathcal{WBGA}$(co) we have
\[
	\lim_{m\to\infty} E(G_m) = \inf_{x\in D} E(x).
\]
\end{Theorem}

Here are two simple corollaries of Theorem~\ref{thm:wbga_conv}.
\begin{Corollary}
Let $E$ be a uniformly smooth on $D \subset X$ convex function.
Then any algorithm from the class $\mathcal{WBGA}$(co) with a constant weakness sequence $\tau = t \in (0,1]$ converges, i.e.
\[
	\lim_{m\to\infty} E(G_m) = \inf_{x\in D} E(x).
\]
\end{Corollary}

\begin{Corollary}
Let $E$ be a convex function with the modulus of smoothness of power type $1 < q \le 2$, that is, $\rho(E,D,u) \le \gamma u^q$.
Let a sequence $\tau := \{t_m\}_{m=1}^\infty$, $t_m \in (0,1]$ for $m = 1,2,3,\ldots$ be such that 
\[
	\sum_{m=1}^\infty t_m^p = \infty, \ \ p = \frac{q}{q-1}.
\]
Then any algorithm from the class $\mathcal{WBGA}$(co) with the weakness sequence $\tau$ converges, i.e.
\[
	\lim_{m \to \infty} E(G_m) = \inf_{x\in D} E(x).
\]
\end{Corollary}

We now proceed to the rate of convergence estimates, which are of interest in both finite dimensional and infinite dimensional settings.
A typical assumption in this regard is formulated in terms of the convex hull $\A_1(\D)$ of the dictionary $\D$, defined by~\eqref{eq:A_1(D)}.
\begin{Theorem}\label{thm:wbga_rate}
Let $E$ be a convex function with the modulus of smoothness of power type $1 < q \le 2$, that is, $\rho(E,D,u) \le \gamma u^q$.
Take an element $f^\e \in D$ and a number $\e \ge 0$ such that
\[
	E(f^\e) \le \inf_{x\in D} E(x) + \e, \ \ f^\e/A(\e) \in \A_1(\D)
\]
with some number $A(\e) \ge 1$.
Then for any algorithm from the class $\mathcal{WBGA}$(co) we have
\[
	E(G_m) - \inf_{x\in D} E(x) \le \max\left\{ 2\e, C(q,\gamma) A(\e)^q \left(C(E,q,\gamma) + \sum_{k=1}^m t_k^p\right)^{1-q} \right\},
\]
with $p = q/(q-1)$ and constants $C(q,\gamma)$ and $C(E,q,\gamma)$.
\end{Theorem}

\begin{Corollary}
Let $E$ be a convex function with the modulus of smoothness of power type $1 < q \le 2$, that is, $\rho(E,D,u) \le \gamma u^q$.
If $\operatorname{argmin}_{x \in D} E(x) \in \A_1(\D)$ then for any algorithm from the class $\mathcal{WBGA}$(co) we have
\[
	E(G_m) - \inf_{x \in D} E(x) \le C(q,\gamma) \left(C(E,q,\gamma) + \sum_{k=1}^m t_k^p\right)^{1-q},
\]
with $p = q/(q-1)$ and constants $C(q,\gamma)$ and $C(E,q,\gamma)$.
\end{Corollary}

\begin{Remark}
While the results stated in this section are known for the WCGA(co) and the WGAFR(co) (see~\cite{VT140}), they are novel for the RWRGA(co).
\end{Remark}

\section{Weak Biorthogonal Greedy Algorithms with errors for Convex Optimization}\label{sec:awbga}
In this section we address the question of the stability of algorithms from the class $\mathcal{WBGA}$(co) by introducing the wider class $\mathcal{WBGA}(\Delta,\text{co})$, which allows for imprecise calculations in the realization of algorithms.
Such an approach is of a practical interest since computational inaccuracies often occur naturally in applications.
To account for imprecise computations we introduce a sequence $\Delta := \{\de_m, \e_m\}_{m=1}^\infty$, where $\de_m \in [0,1]$ and $\e_m \ge 0$ for $m = 1,2,3,\dots$, that represents the allowed inaccuracies in the steps of the algorithms.
In accordance with the conventional notation (see e.g.~\cite{gribonval2001approximate}, \cite{galatenko2003convergence}), we refer to a given sequence $\Delta := \{\de_m, \e_m\}_{m=1}^\infty$ as an {\it error sequence}.

For a convex Fr{\'e}chet-differentiable target function $E : X \to \mathbb{R}$ we define the following class of greedy algorithms with errors.
\\[.5em]\noindent
{\bf Weak Biorthogonal Greedy Algorithms with errors ($\boldsymbol{\mathcal{WBGA}(\Delta,\text{co})}$).\\} 
We say that an algorithm belongs to the class $\mathcal{WBGA}(\Delta,\text{co})$ with a weakness sequence $\tau = \{t_m\}_{m=1}^\infty$, $t_m\in[0,1]$ and an error sequence $\Delta = \{\de_m,\e_m\}_{m=1}^\infty$, $\de_m\in[0,1], \e_m\ge0$, if sequences of approximators $\{G_m\}_{m=0}^\infty$ and selected elements $\{\varphi_m\}_{m=1}^\infty$ of the dictionary $\D$ satisfy the following conditions at every iteration $m \ge 1$:
\begin{enumerate}[label=\bf(\arabic*), leftmargin=.5in]
	\item\label{awbga_gs}
		Greedy selection: ${\displaystyle \<-E'(G_{m-1}), \varphi_m\> \ge t_m \sup_{\varphi\in\D} \<-E'(G_{m-1}), \varphi\>}$;
	\item\label{awbga_er}
		Error reduction: ${\displaystyle E(G_m) \le \inf_{\la\ge0} E(G_{m-1} + \la\ff_m) + \de_m}$;
	\item\label{awbga_bo}
		Biorthogonality: ${\displaystyle |\<E'(G_{m}), G_m\>| \le \e_m}$;${\displaystyle\phantom{\inf_\la}}$
	\item\label{awbga_bd}
		Boundedness: ${\displaystyle E(G_{m}) \le E(0) + C_0}$.
\end{enumerate}

Note that in addition to conditions~\ref{wbga_gs}--\ref{wbga_bo} from the definition of the class $\mathcal{WBGA}$(co), for the $\mathcal{WBGA}(\Delta,\text{co})$ we require the boundedness condition~\ref{awbga_bd} to account for the magnitude of allowed errors $\Delta$.
In particular, if the error sequence $\Delta$ is summable, i.e. $\sum_{m=1}^\infty \de_m < \infty$, then condition~\ref{awbga_bd} follows directly from~\ref{awbga_er} with $C_0 = \sum_{m=1}^\infty \de_m$.

Moreover, we assume that the set
\[
	D \subset D_1 := \{x \in X : E(x) \le E(0) + C_0\} \subset X,
\]
where $C_0$ is the constant from condition~\ref{awbga_bd}, is bounded.
Then condition~\ref{awbga_bd} guarantees that $G_m \in D_1$ for all $m \ge 0$ for any algorithm from the $\mathcal{WBGA}(\Delta,\text{co})$.

\begin{Remark}
In the error reduction condition~\ref{awbga_er} from the definition of the class $\mathcal{WBGA}(\Delta,co$) the infimum is taken over all $\la \ge 0$.
In order to simplify this problem, one can consider a wider than the $\mathcal{WBGA}(\Delta,co)$ class~--- the class $\mathcal{WBGA}(\Delta,[0,1],co)$ of algorithms satisfying conditions~\ref{awbga_gs}, \ref{awbga_bo}, \ref{awbga_bd}, and the following condition instead of~\ref{awbga_er}:
\[
	\text{{\bf (2')} {\rm Restricted error reduction: }}
	E(G_m) \le \inf_{\la\in[0,1]} E(G_{m-1} + \la\ff_m) + \de_m.
\]
Then finding such $\la\in[0,1]$ is a line search problem, which is known to be a simple one-dimensional convex optimization problem (see e.g.~\cite{boyd2004convex}, \cite{N}).
\end{Remark}

\subsection{Examples of algorithms from the $\mathcal{WBGA}(\Delta,co)$}\label{sec:awbga_ga}
In this section we briefly overview particular algorithms from the class $\mathcal{WBGA}(\Delta,\text{co})$ that correspond to the approximate versions of the algorithms considered in Section~\ref{sec:wbga_ga}.
Denote by $\tau := \{t_m\}_{m=1}^\infty$ and $\Delta := \{\de_m,\e_m\}_{m=1}^\infty$ a weakness sequence and an error sequence respectively, i.e. given sequences of numbers $t_m \in [0,1]$, $\de_m \in [0,1]$, and $\e_m \ge 0$ for $m = 1,2,3,\ldots$.

We begin with the Weak Chebyshev Greedy Algorithm with errors for optimization of convex objective functions. 
\\[.5em]\noindent
{\bf Weak Chebyshev Greedy Algorithm with errors (WCGA($\Delta,\text{co}$)).\\}
Set $G^c_0 = 0$ and for each $m \ge 1$ perform the following steps:
\begin{enumerate}
	\item Take any $\varphi^{c}_m \in \D$ satisfying
		\[
			\< -E'(G^c_{m-1}), \varphi^c_m \> \ge t_m \sup_{\varphi\in\D} \< -E'(G^c_{m-1}), \varphi \>;
		\]
	\item Denote $\Phi_m^c = \spn \{\varphi^c_k\}_{k=1}^m$ and find $G_m^c \in \Phi_m^c$ such that
		\[
			E(G_m^c) \le \inf_{G \in \Phi_m^c} E(G) + \de_m.
		\]
\end{enumerate}

Next, we state the Weak Greedy Algorithm with Free Relaxation and errors for optimization of convex objective functions, introduced and studied in~\cite{DT}.
\\[.5em]\noindent
{\bf Weak Greedy Algorithm with Free Relaxation and errors (WGAFR($\Delta,\text{co}$)).\\}
Set $G^f_0 = 0$ and for each $m \ge 1$ perform the following steps:
\begin{enumerate}
	\item Take any $\varphi^f_m \in \D$ satisfying
		\[
			\< -E'(G^f_{m-1}), \varphi^f_m \> \ge t_m \sup_{\varphi\in\D} \< -E'(G^f_{m-1}), \varphi \>;
		\]
	\item Find $\omega_m \in \mathbb{R}$ and $ \lambda_m \in \mathbb{R}$ such that
		\[
			E\big( (1-\omega_m) G^f_{m-1} + \lambda_m \varphi^f_m \big)
			\le \inf_{\la, \omega \in \mathbb{R}} E\big( (1 - \omega) G^f_{m-1} + \la \varphi^f_m \big) + \de_m
		\]
		and define $G^f_m = (1 - \omega_m) G^f_{m-1} + \la_m \varphi^f_m$.
\end{enumerate}

Lastly, we introduce a new algorithm~--- the Rescaled Weak Relaxed Greedy Algorithm with errors for optimization of convex objective functions.
\\[.5em]\noindent
{\bf Rescaled Weak Relaxed Greedy Algorithm with errors (RWRGA($\Delta,\text{co}$)).\\}
Set $G^r_0 = 0$ and for each $m \ge 1$ perform the following steps:
\begin{enumerate}
	\item Take any $\varphi^r_m \in \D$ satisfying
		\[
			\< -E'(G^r_{m-1}), \varphi^r_m \> \ge t_m \sup_{\varphi\in\D} \< -E'(G^r_{m-1}), \varphi \>;
		\]
	\item Find $\lambda_m \ge 0$ such that
		\[
			E(G^r_{m-1} + \la_m \varphi^r_m) \le \inf_{\la \ge 0} E(G^r_{m-1} + \la \varphi^r_m) + \de_m/2;
		\]
	\item Find $\mu_m \in \mathbb{R}$ such that
		\[
			E\big( \mu_m (G^r_{m-1} + \la_m \varphi^r_m) \big)
			\le \inf_{\mu \in \mathbb{R}} E\big( \mu (G^r_{m-1} + \la_m \varphi^r_m) \big) + \de_m/2
		\]
		and define $G^r_m = \mu_m (G^r_{m-1} + \la_m \varphi^r_m)$.
\end{enumerate}

\begin{Proposition}\label{prp:ga_awbga}
The WCGA($\Delta$,co), the WGAFR($\Delta$,co), and the RWRGA($\Delta$,co) belong to the class $\mathcal{WBGA}(\Delta,co)$ with
\[
	\e_m = \inf_{u > 0} \frac{\de_m + 2\rho(E,D_1,u \|G_m\|)}{u}.
\]
\end{Proposition}

\subsection{Convergence results for the $\mathcal{WBGA}(\Delta,co)$}
In this section we discuss the convergence and rate of convergence results for algorithms from the class $\mathcal{WBGA}(\Delta,\text{co})$.

First, we state the convergence result.
\begin{Theorem}\label{thm:awbga_conv}
Let $E$ be a uniformly smooth on $D_1 \subset X$ convex function.
Assume that an error sequence $\Delta := \{\de_m,\e_m\}_{m=1}^\infty$ is such that $\de_m \to 0$ and $\e_m \to 0$ as $m \to \infty$.
Then any algorithm from the class $\mathcal{WBGA}(\Delta,\text{co})$ with a constant weakness sequence $\tau = t \in (0,1]$ converges, i.e.
\[
	\lim_{m\to\infty} E(G_m) = \inf_{x\in D_1} E(x).
\]
\end{Theorem}

Second, we provide the rate of convergence estimate.
\begin{Theorem}\label{thm:awbga_rate}
Let $E$ be a convex function with the modulus of smoothness of power type $1 < q \le 2$, that is, $\rho(E,D_{1},u) \le \gamma u^q$.
Take an element $f^\e \in D_1$ and a number $\e \ge 0$ such that
\[
	E(f^\e) \le \inf_{x \in D_1} E(x) + \e, \ \ 
	f^\e/A \in \A_1(\D),
\]
with some number $A := A(\e) \ge 1$.
Then for any algorithm from the class $\mathcal{WBGA}(\Delta,\text{co})$ with a constant weakness sequence $\tau = t \in (0,1]$ and an error sequence $\Delta = \{\de_m,\e_m\}_{m=1}^\infty$ with $\de_m + \e_m \le cm^{-q}$, $m = 1,2,3,\dots$ we have
\[
	E(G_m) - \inf_{x\in D_1} E(x) \le \e + C(E,q,\gamma,t,c) A(\e)^q \, m^{1-q}.
\]
\end{Theorem}

\begin{Corollary}
Under the conditions of Theorem~\ref{thm:awbga_rate}, specifying 
\[
	A(\e) := \inf \Big\{ A > 0 : \exists f \in D_1 : f/A \in \A_1(\D),\ \ E(f) \le \inf_{x\in D_1}E(x) + \e \Big\}
\]
and denoting
\[
	\eta_m := \inf \big\{ \e > 0: A(\e)^q \, m^{1-q} \le \e \big\},
\]
we obtain for any algorithm from the class $\mathcal{WBGA}(\Delta,co)$
\[
	E(G_m) - \inf_{x\in D_1} E(x) \le C(E,q,\gamma,t) \, \eta_m.
\]
\end{Corollary}

\begin{Remark}
It follows from the proofs of Theorems~\ref{thm:awbga_conv} and~\ref{thm:awbga_rate}, given in Section~\ref{sec:proofs_awbga}, that the results stated in this section also hold for the class $\mathcal{WBGA}(\Delta,[0,1],\text{co})$.
\end{Remark}

\section{Numerical experiments}\label{sec:numerics}
In this section we demonstrate the performance of the algorithms from the class $\mathcal{WBGA}$(co) that are discussed in Section~\ref{sec:wbga_ga}: the Weak Chebyshev Greedy Algorithm (WCGA(co)), the Weak Greedy Algorithm with Free Relaxation (WGAFR(co)), and the Rescaled Weak Relaxed Greedy Algorithm (RWRGA(co)).

For each of the numerical experiments presented below we consider the Banach space $X = \ell_p^{(\mathrm{dim})}$ of dimensionality $\mathrm{dim}$, a target function $E : X \to \mathbb{R}$, and a dictionary $\D \in X$.
We then employ the aforementioned algorithms to solve the optimization problem~\eqref{eq:opt}, i.e. to find an approximate sparse (with respect to the dictionary $\D$) minimizer
\[
	x^* = \mathop{\operatorname{argmin}}_{x \in X} E(x).
\]
Since greedy algorithms are iterative by design, in Examples~1--2 we obtain and present the trade-off between the sparsity of the solution $x^*$ and the value of $E(x^*)$.
In Examples~3--4, we additionally compare the greedy algorithms for convex optimization with a conventional method of finding sparse solutions~--- the optimization with $\ell_1$-regularization, see~\eqref{eq:opt_reg}.
Specifically, we solve the problem
\[
	\text{find}\ \ x^* = \mathop{\operatorname{argmin}}_{x \in X} \Big( E(x) + \lambda \|x\|_\D \Big),
\]
where $\|\cdot\|_\D$ is the atomic norm with respect to the dictionary $\D$, defined by~\eqref{eq:D_norm}.
To obtain approximate minimizers of different sparsities, the values of the regularization parameter $\lambda$ are taken from the sequence $\{0.1 \times (0.9)^k\}_{k=0}^{49}$, i.e. $50$ regularized optimization problems are solved in every setting.

To avoid an unintentional bias in the selection of the space $X$, the dictionary $\D$, and the target function $E$, we generate those randomly, based on certain parameters that are described in the setting of each example.
In order to provide a reliable demonstration that is independent of a particular random generation, we compute $100$ simulations for each presented example and provide the distribution of the optimization results (shown in Figures~1--4).
In the presented pictures the solid lines represent the geometric mean of the function values for each algorithm and the filled areas represent the minimization distribution across all $100$ simulations.
Finally, to make the results consistent across simulations, we rescale the optimization results to be in the interval $[0,1]$, i.e. instead of reporting the value of $E(x^*)$ we report
\[
	\frac{E(x^*) - \inf_{x \in X} E(x)}{E(0) - \inf_{x \in X} E(x)} \in [0,1].
\]

Numerical experiments presented in this section are performed in Python~3.6 with the use of NumPy and SciPy libraries.
The source code is available at~\url{https://github.com/sukiboo/wbga_co}.

\subsection{Example 1}
\begin{figure}[ht!]
	\includegraphics[width=.49\linewidth]{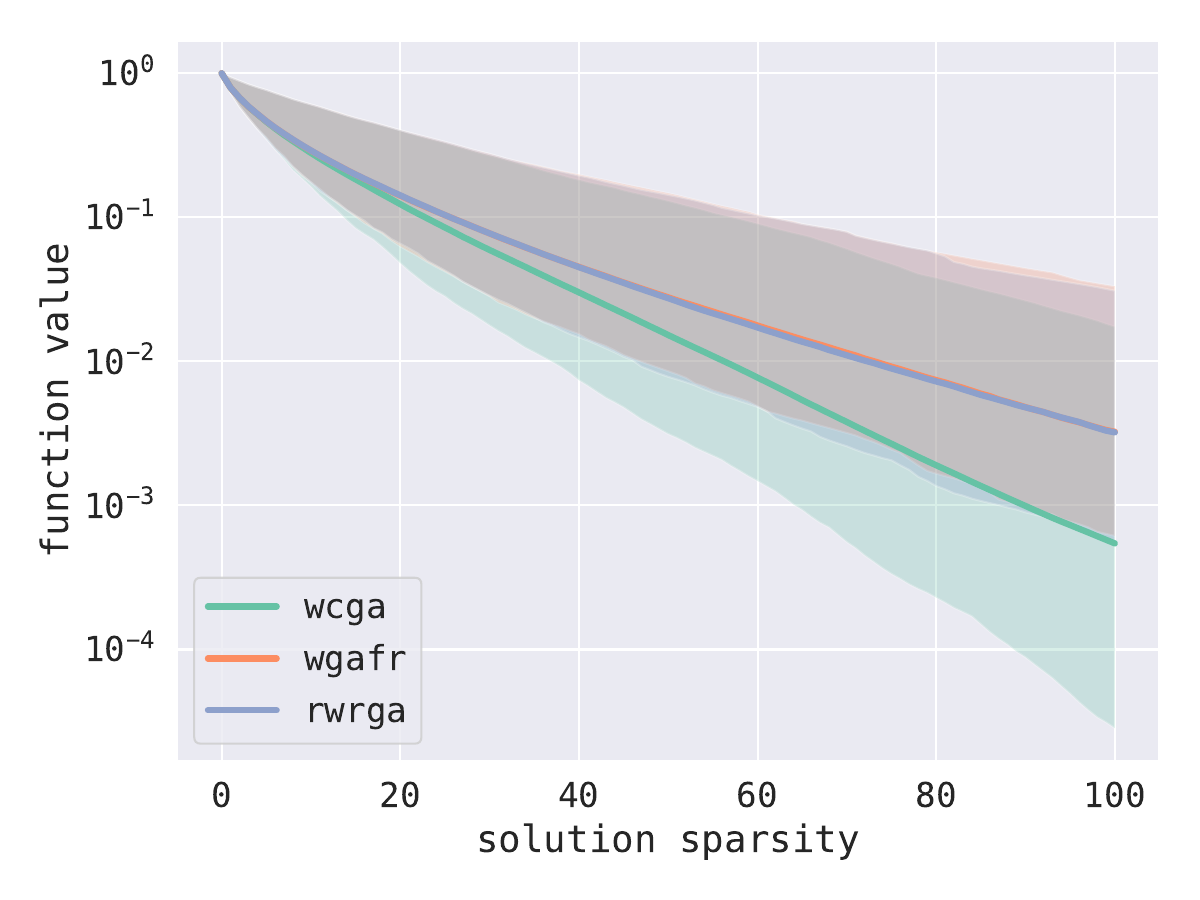}
	\includegraphics[width=.49\linewidth]{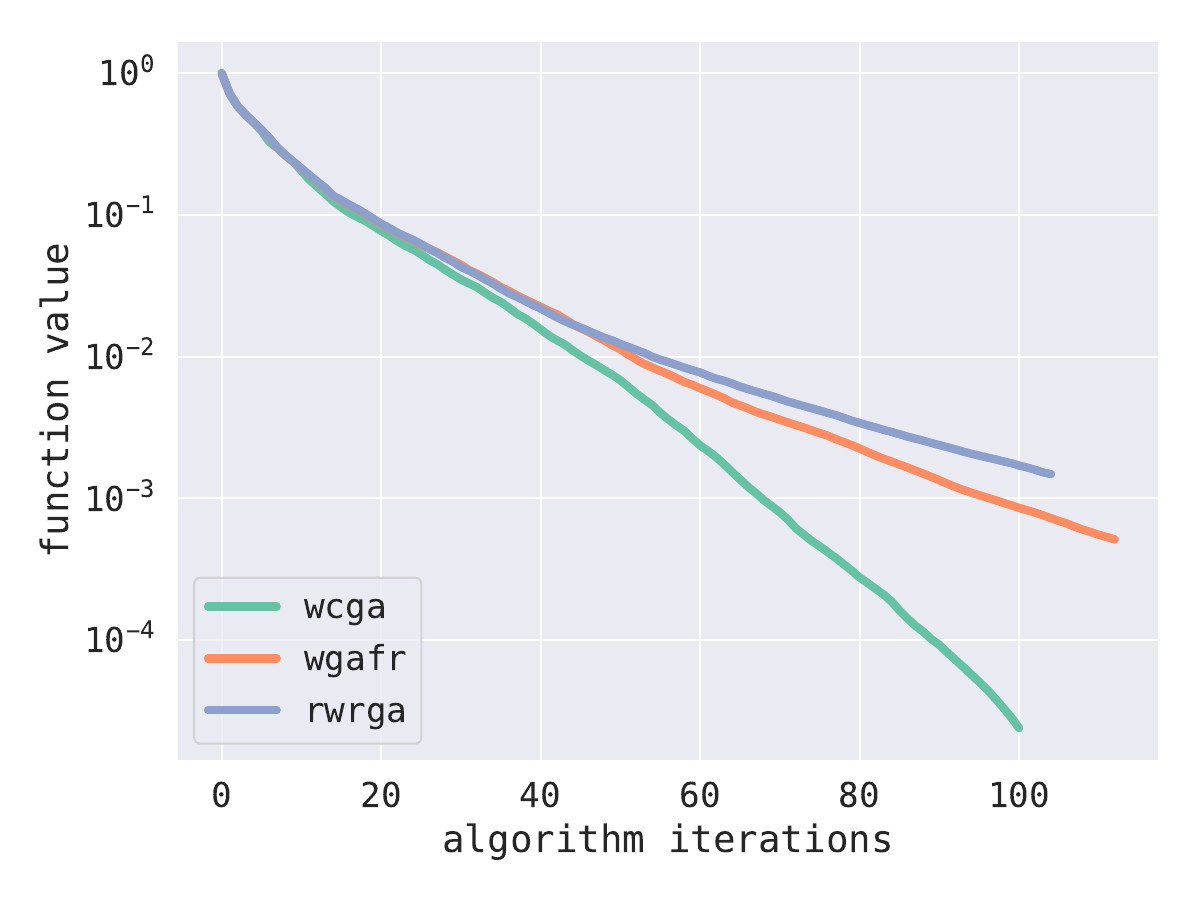}
	\caption{Optimization results for Example~1: function values vs solution sparsity (left) and a particular realization of the minimization process vs the algorithm iterations (right).}
\end{figure}
In this example we consider the space $X = \ell_p^{(1000)}$ with $p \sim \mathcal{U}(1,10)$, and a dictionary $\D$ of size $10000$, constructed as normalized linear combinations of the canonical basis $\{e_i\}_{i=1}^{1000}$ of $X$ with normally distributed coefficients, i.e.
\[
	\D = \{\varphi_j\}_{j=1}^{10000},
	\ \text{where}\ 
	\varphi_j = \frac{\sum_{i=1}^{1000} c^i_j \, e_i}{\left( \sum_{i=1}^{1000} |c_j^i|^p \right)^{1/p}}
	\ \text{with}\ 
	c^i_j \sim \mathcal{N}(0,1).
\]
The target function $E : X \to \mathbb{R}$ is chosen as
\[
	E(x) = \|x - f\|_q^q,
\]
where $q \sim \mathcal{U}(1,10)$ is randomly sampled and $f \in X$ is randomly generated as a linear combination of $K_f \sim \mathcal{U}(100,300)$ randomly selected elements of $\D$ with normally distributed coefficients, i.e.
\begin{gather*}
	f = \sum_{k=1}^{K_f} a_k \, \varphi_{\sigma(k)},
	\ \text{where}\ 
	a_k \sim \mathcal{N}(0,1)
	\\
	\ \text{and}\ 
	\sigma\ \text{is a permutation of}\ \{1, \ldots, 10000\}.
\end{gather*}
We deploy the algorithms to obtain an approximate minimizer of sparsity $100$.
Performance of the greedy algorithms over $100$ simulations in this setting is presented in Figure~1.

\subsection{Example 2}
\begin{figure}[ht!]
	\includegraphics[width=.49\linewidth]{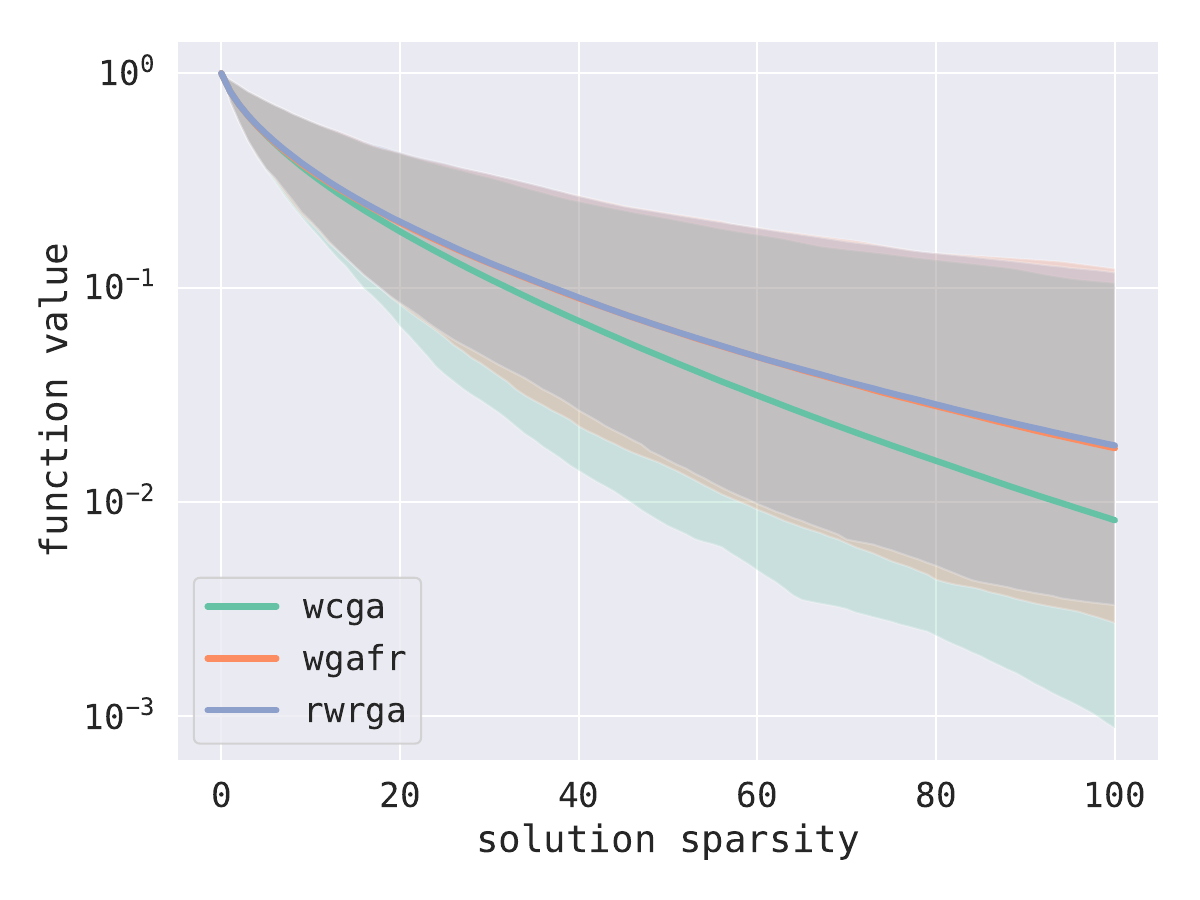}
	\includegraphics[width=.49\linewidth]{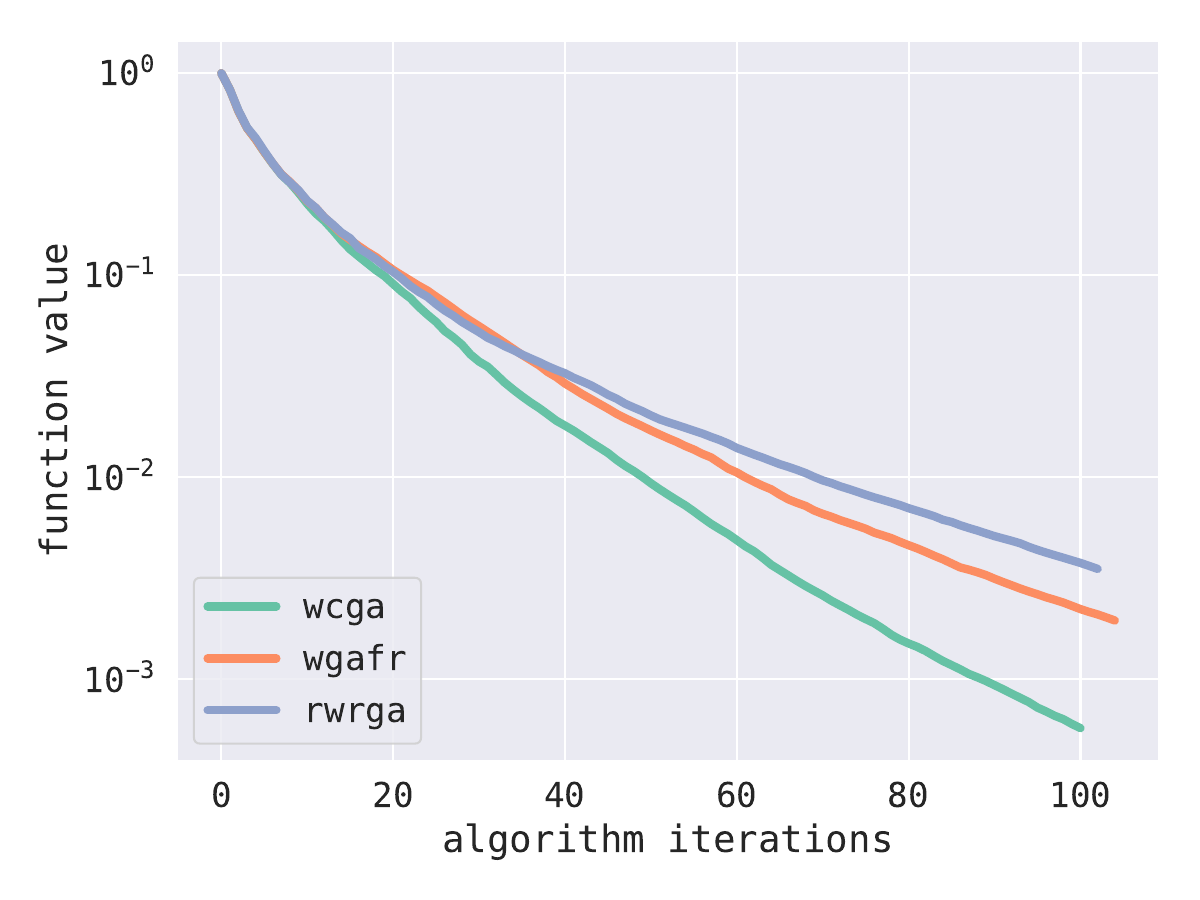}
	\caption{Optimization results for Example~2: function values vs solution sparsity (left) and a particular realization of the minimization process vs the algorithm iterations (right).}
\end{figure}
In this example we consider the space $X = \ell_p^{(1000)}$ with $p \sim \mathcal{U}(1,10)$ and a dictionary $\D$ of size $10000$, constructed as normalized linear combinations of the canonical basis $\{e_i\}_{i=1}^{1000}$ of $X$ with normally distributed coefficients, i.e.
\[
	\D = \{\varphi_j\}_{j=1}^{10000},
	\ \text{where}\ 
	\varphi_j = \frac{\sum_{i=1}^{1000} c^i_j \, e_i}{\left( \sum_{i=1}^{1000} |c_j^i|^p \right)^{1/p}}
	\ \text{with}\ 
	c^i_j \sim \mathcal{N}(0,1).
\]
The target function $E : X \to \mathbb{R}$ is chosen as
\[
	E(x) = \|x - f\|_q^q + \|x - g\|_r^r,
\]
where $q \sim \mathcal{U}(2,10)$ and $r \sim \mathcal{U}(1,2)$, and the elements $f,g \in X$ are each randomly generated as linear combinations of $K_f,K_g \sim \mathcal{U}(200,400)$ randomly selected elements of $\D$ with normally distributed coefficients, i.e.
\begin{gather*}
	f = \sum_{k=1}^{K_f} a^1_k \, \varphi_{\sigma_1(k)}
	\ \text{and}\ 
	g = \sum_{k=1}^{K_g} a^2_k \, \varphi_{\sigma_2(k)},
	\\
	\text{where}\ 
	a^1_k, a^2_k \sim \mathcal{N}(0,1)
	\ \text{and}\ 
	\sigma_1, \sigma_2\ \text{are permutations of}\ \{1,\ldots,10000\}.
\end{gather*}
We deploy the algorithms to obtain an approximate minimizer of sparsity $100$.
Performance of the greedy algorithms over $100$ simulations in this setting is presented in Figure~2.

\subsection{Example 3}
\begin{figure}[ht!]
	\includegraphics[width=.49\linewidth]{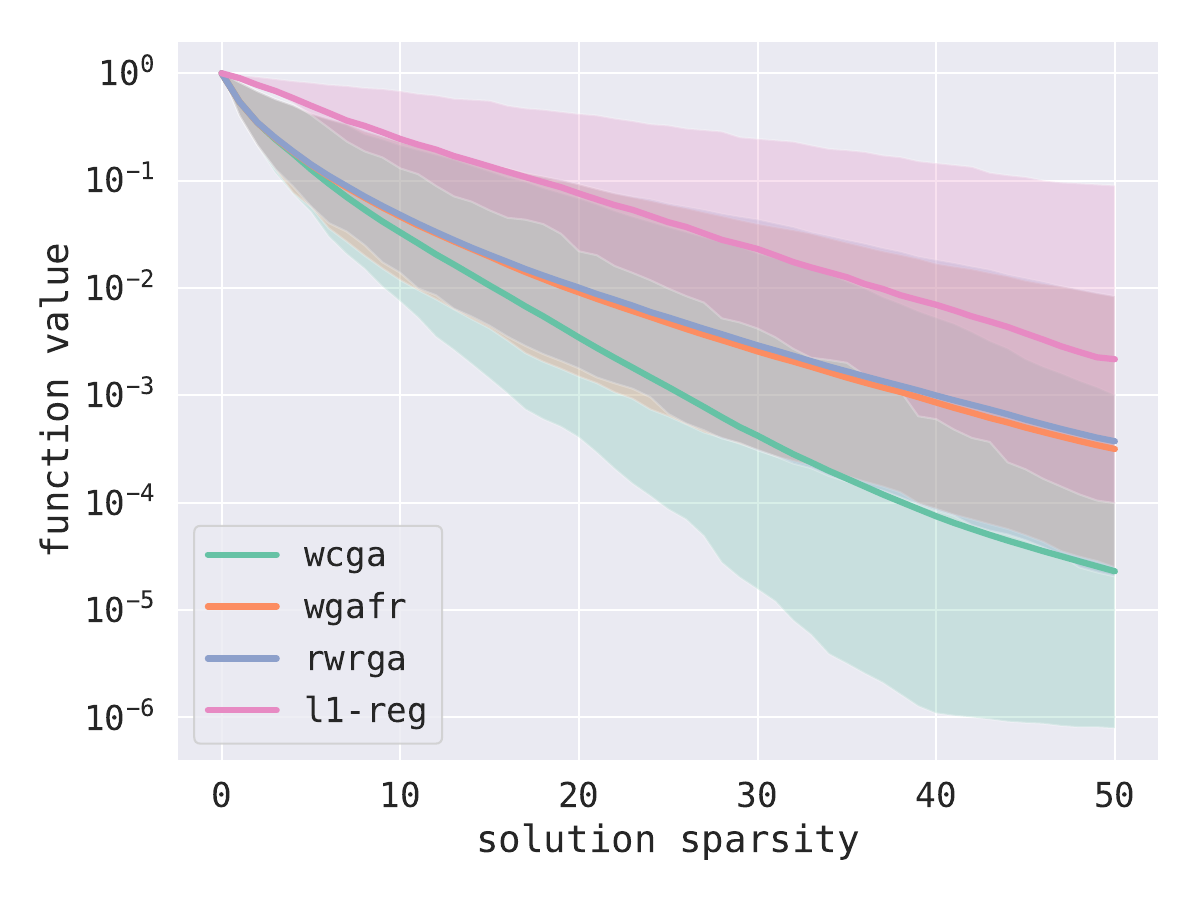}
	\includegraphics[width=.49\linewidth]{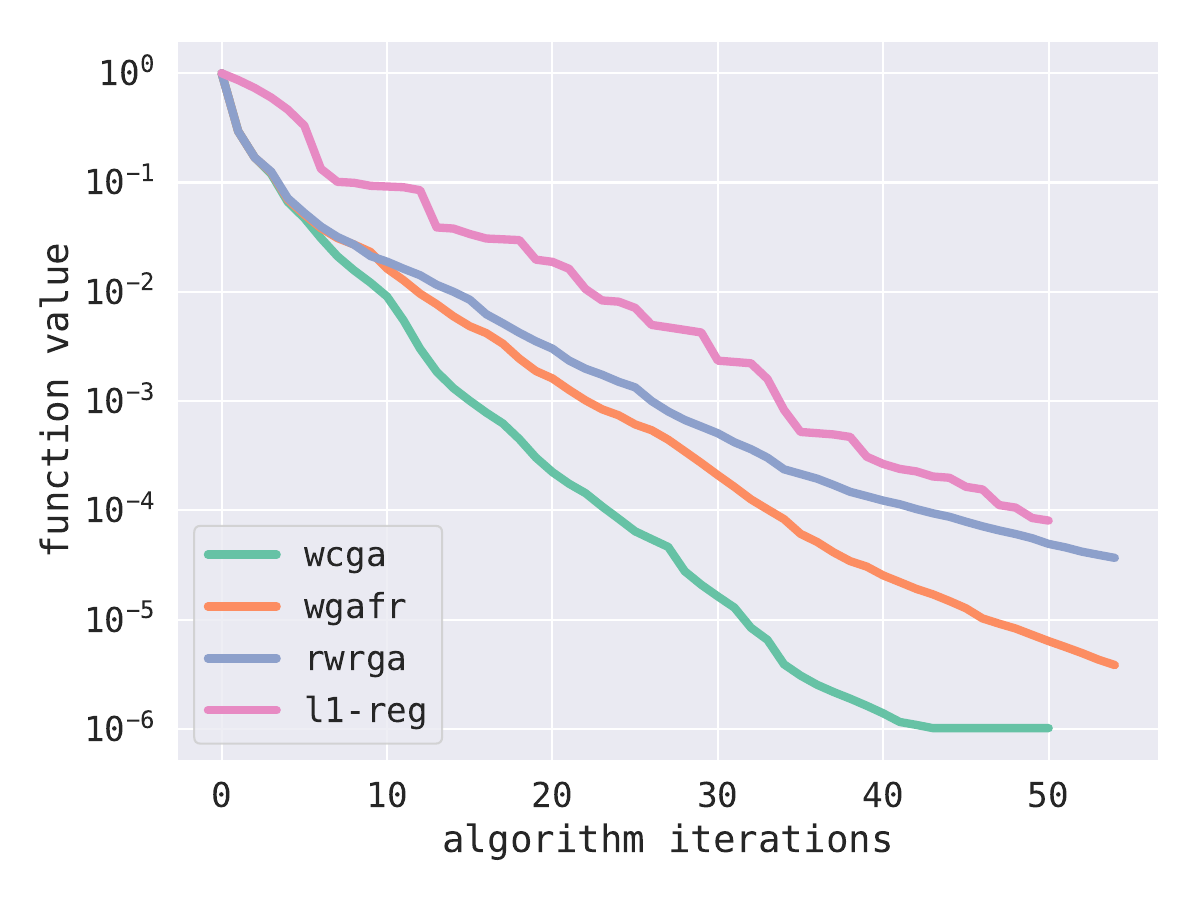}
	\caption{Optimization results for Example~3: function values vs solution sparsity (left) and a particular realization of the minimization process vs the algorithm iterations (right).}
\end{figure}
In this example we additionally compare the greedy algorithms with conventional optimization with $\ell_1$-regularization, see~\eqref{eq:opt_reg}.
Since obtaining the minimization-sparsity trade-off with $\ell_1$-regularization is more expensive computationally than it is for the greedy algorithms, we restrict ourselves to work in a space of smaller dimensionality.
Namely, we consider the space $X = \ell_p^{(100)}$ with $p \sim \mathcal{U}(1,10)$, and a dictionary $\D$ of size $500$, constructed as normalized linear combinations of the canonical basis $\{e_i\}_{i=1}^{100}$ of $X$ with normally distributed coefficients, i.e.
\[
	\D = \{\varphi_j\}_{j=1}^{500},
	\ \text{where}\ 
	\varphi_j = \frac{\sum_{i=1}^{100} c^i_j \, e_i}{\left( \sum_{i=1}^{100} |c_j^i|^p \right)^{1/p}}
	\ \text{with}\ 
	c^i_j \sim \mathcal{N}(0,1).
\]
The target function $E : X \to \mathbb{R}$ is chosen as
\[
	E(x) = \|x - f\|_q^q,
\]
where $q \sim \mathcal{U}(1,10)$ is randomly sampled and $f \in X$ is randomly generated as a linear combination of $K_f \sim \mathcal{U}(50,100)$ randomly selected elements of $\D$ with normally distributed coefficients, i.e.
\begin{gather*}
	f = \sum_{k=1}^{K_f} a_k \, \varphi_{\sigma(k)},
	\ \text{where}\ 
	a_k \sim \mathcal{N}(0,1)
	\\
	\ \text{and}\ 
	\sigma\ \text{is a permutation of}\ \{1, \ldots, 500\}.
\end{gather*}
We deploy the algorithms to obtain an approximate minimizer of sparsity $50$.
Performance of the greedy algorithms and optimization with $\ell_1$-regularization over $100$ simulations in this setting is presented in Figure~3.

\subsection{Example 4}
\begin{figure}[ht!]
	\includegraphics[width=.49\linewidth]{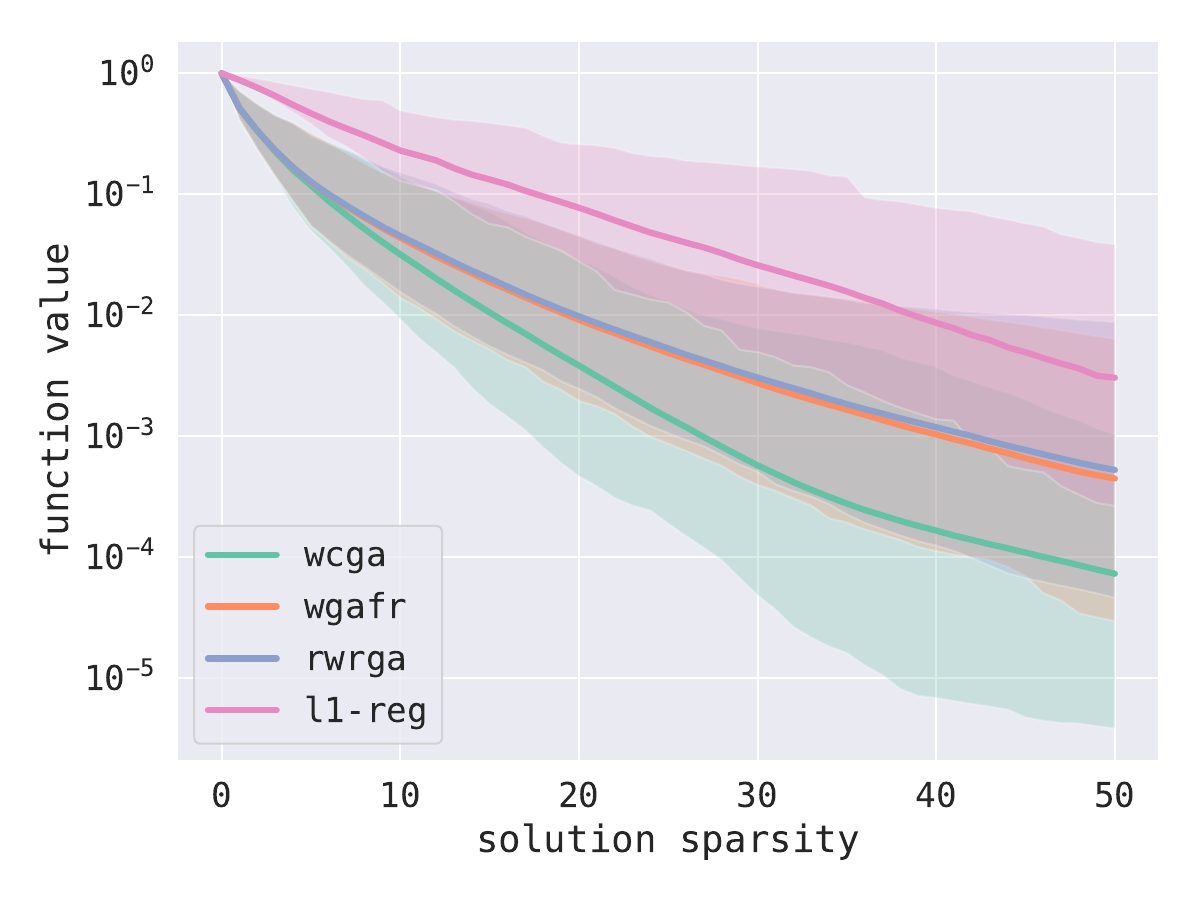}
	\includegraphics[width=.49\linewidth]{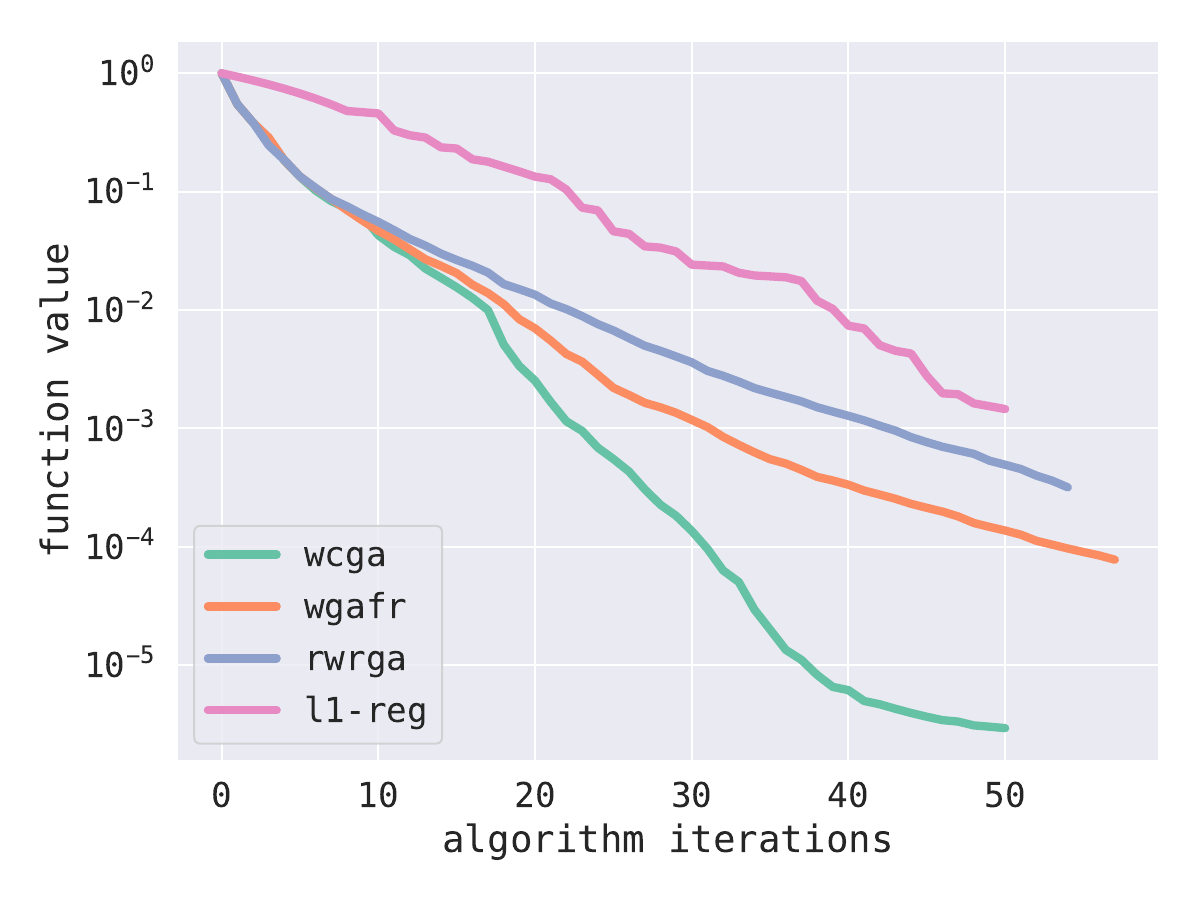}
	\caption{Optimization results for Example~4: function values vs solution sparsity (left) and a particular realization of the minimization process vs the algorithm iterations (right).}
\end{figure}
In this example we consider the space $X = \ell_p^{(100)}$ with $p \sim \mathcal{U}(1,10)$, and a dictionary $\D$ of size $500$, constructed as normalized linear combinations of the canonical basis $\{e_i\}_{i=1}^{100}$ of $X$ with normally distributed coefficients, i.e.
\[
	\D = \{\varphi_j\}_{j=1}^{500},
	\ \text{where}\ 
	\varphi_j = \frac{\sum_{i=1}^{100} c^i_j \, e_i}{\left( \sum_{i=1}^{100} |c_j^i|^p \right)^{1/p}}
	\ \text{with}\ 
	c^i_j \sim \mathcal{N}(0,1).
\]
The target function $E : X \to \mathbb{R}$ is chosen as
\[
	E(x) = \|x - f\|_q^q + \|x - g\|_r^r,
\]
where $q \sim \mathcal{U}(2,10)$ and $r \sim \mathcal{U}(1,2)$, and the elements $f,g \in X$ are each randomly generated as linear combinations of $K_f,K_g \sim \mathcal{U}(100,200)$ randomly selected elements of $\D$ with normally distributed coefficients, i.e.
\begin{gather*}
	f = \sum_{k=1}^{K_f} a^1_k \, \varphi_{\sigma_1(k)}
	\ \text{and}\ 
	g = \sum_{k=1}^{K_g} a^2_k \, \varphi_{\sigma_2(k)},
	\\
	\text{where}\ 
	a^1_k, a^2_k \sim \mathcal{N}(0,1)
	\ \text{and}\ 
	\sigma_1, \sigma_2\ \text{are permutations of}\ \{1,\ldots,500\}.
\end{gather*}
We deploy the algorithms to obtain an approximate minimizer of sparsity $50$.
Performance of the greedy algorithms and optimization with $\ell_1$-regularization over $100$ simulations in this setting is presented in Figure~4.

\section{Proofs for Section~\ref{sec:wbga}}\label{sec:proofs_wbga}
In this section we provide the proofs of the results from Section~\ref{sec:wbga}.
We begin with a known lemma.
\begin{Lemma}[{\cite[Lemma~6.1]{VT140}}]\label{lem:E'_L=0}
Let $E$ be a uniformly smooth Fr{\'e}chet-differentiable convex function on a Banach space $X$ and $L$ be a finite-dimensional subspace of $X$.
Let $x_L$ denote the point from $L$ at which $E$ attains the minimum, i.e.
\[
	x_L = \mathop{\operatorname{argmin}}_{x \in L} E(x) \in L.
\]
Then for any $\phi \in L$ we have
\[
	\<E'(x_L), \phi\> = 0.
\]
\end{Lemma}

We now prove that the algorithms stated in Section~\ref{sec:wbga_ga} belong to the class $\mathcal{WBGA}$(co).
\begin{proof}[Proof of Proposition~\ref{prp:ga_wbga}]
It is easy to see that conditions~\ref{wbga_gs} and~\ref{wbga_er} from the definition of the class $\mathcal{WBGA}$(co) are satisfied for all three algorithms.
Condition~\ref{wbga_bo} for any $m \ge 1$ follows directly from Lemma~\ref{lem:E'_L=0} with $x_L = \phi = G_m$ and
\[
	L = \Phi^c_m = \spn\{\varphi_1^c, \ldots, \varphi_m^c\}
\]
for the WCGA(co), and
\[
	L = \spn\{G_{m-1}^f, \varphi_m^f\}
	\ \ \text{or}\ \ 
	L = \spn\{G_{m-1}^r, \varphi_m^r\}
\]
for the WGAFR(co) / RWRGA(co) respectively.
\end{proof}

We proceed by listing the lemmas that will be utilized later in the proofs of the main results.
The following simple lemma is well-known (see, for instance, \cite{VT140}).
For the reader's convenience we present its proof here.
\begin{Lemma}[{\cite[Lemma~6.3]{VT140}}]\label{lem:E'_rho}
Let $E$ be a Fr{\'e}chet-differentiable convex function.
Then the following inequality holds for any $x \in S \subset X$, $y \in X$, and $u \in \mathbb{R}$
\[
	0 \le E(x + uy) - E(x) - u\<E'(x),y\> \le 2\rho(E,S,u\|y\|).
\]
\end{Lemma}
\begin{proof}
The left inequality follows directly from~\eqref{eq:E'_conv1}.
Next, from the definition of modulus of smoothness~\eqref{eq:mod_smt} it follows that
\[
	E(x + uy) + E(x - uy) \le 2\big( E(x) + \rho(E,S,u\|y\|) \big).
\]
From inequality \eqref{eq:E'_conv1} we get
\[
	E(x - uy) \ge E(x) - u\<E'(x),y\>. 
\]
Combining the above two estimates, we obtain
\[
	E(x + uy) \le E(x) + u\<E'(x),y\> + 2\rho(E,S,u\|y\|),
\]
which proves the second inequality. 
\end{proof}

\begin{Lemma}[{\cite[Lemma~6.10]{VTbook}}]\label{lem:F_A1(D)}
For any bounded linear functional $F$ and any dictionary $\D$, we have
\[
	\sup_{g\in \D} \<F,g\> = \sup_{f\in\A_1(\D)} \<F,f\>.
\]
\end{Lemma}

The following lemma is similar to the result from~\cite{T13}.
For the reader's convenience we present a brief proof of this lemma here.
\begin{Lemma}\label{lem:y_k}
Suppose that a sequence $y_1 \ge y_2 \ge y_3 \ge \ldots > 0$ satisfies inequalities
\[
	y_k \le y_{k-1} (1 - w_k y_{k-1}), \ \ w_k \ge 0
\]
for any $k > n$.
Then for any $m > n$ we have
\[
	\frac{1}{y_m} \ge \frac{1}{y_n} + \sum_{k=n+1}^m w_k.
\]
\end{Lemma}
\begin{proof}
The proof follows directly from the chain of inequalities
\[
	\frac{1}{y_k} \ge \frac{1}{y_{k-1}} (1 - w_k y_{k-1})^{-1}
	\ge \frac{1}{y_{k-1}} (1 + w_k y_{k-1})
	= \frac{1}{y_{k-1}} + w_k.
\]
\end{proof}

The following lemma is our key tool for establishing convergence and rate of convergence of algorithms from the class $\mathcal{WBGA}$(co).
\begin{Lemma}[{{\bf Error Reduction Lemma}}]\label{lem:erl}
Let $E$ be a uniformly smooth on $D \subset X$ convex function with the modulus of smoothness $\rho(E,D,u)$.
Take a number $\e\ge 0$ and an element $f^\e \in D$ such that
\[
	E(f^\e) \le \inf_{x \in X} E(x) + \e, \ \ 
	f^\e / A \in \A_1(\D),
\]
with some number $A := A(\e) \ge 1$.
Then for any algorithm from the class $\mathcal{WBGA}$(co) we have for any $m \ge 1$
\begin{multline*}
	E(G_m) - E(f^\e) \le E(G_{m-1}) - E(f^\e)
	\\
	+ \inf_{\la\ge0} \Big(-\la t_m A^{-1} (E(G_{m-1})-E(f^\e)) + 2\rho(E,D,\la) \Big).
\end{multline*}
\end{Lemma}
\begin{proof}
The main idea of the proof is the same as in the proof of the corresponding one-step improvement inequality for the WCGA (see, for instance, \cite[Lemma~6.11]{VTbook}).
It follows from~\ref{wbga_er} of the definition of the class $\mathcal{WBGA}$(co) that
\[
	E(0) \ge E(G_1) \ge E(G_2) \ldots.
\]
Thus if $E(G_{m-1}) - E(f^\e) \le 0$ then the claim of Lemma~\ref{lem:erl} is trivial.
Assuming $E(G_{m-1}) - E(f^\e) > 0$, Lemma~\ref{lem:E'_rho} provides for any $\la \ge 0$
\[
	E(G_{m-1} + \la \varphi_m) \le E(G_{m-1}) - \la \<-E'(G_{m-1}),\varphi_m\> + 2 \rho(E,D,\la)
\]
and by~\ref{wbga_gs} from the definition of the class $\mathcal{WBGA}$(co) and Lemma~\ref{lem:F_A1(D)} we get
\begin{align*}
	\<-E'(G_{m-1}),\varphi_m\> 
	&\ge t_m \sup_{g\in \D} \<-E'(G_{m-1}),g\> 
	\\
	&= t_m\sup_{\phi \in \A_1(\D)} \<-E'(G_{m-1}),\phi\>
	\ge t_m A^{-1} \<-E'(G_{m-1}),f^\e\>.
\end{align*}
By~\ref{wbga_bo} from the definition of the class $\mathcal{WBGA}$(co) and by convexity~\eqref{eq:E'_conv2} we obtain
\[
	\<-E'(G_{m-1}),f^\e\> = \<-E'(G_{m-1}),f^\e-G_{m-1}\> \ge E(G_{m-1})-E(f^\e).
\]
Thus, by~\ref{wbga_er} from the definition of the $\mathcal{WBGA}$(co) we deduce
\begin{align*}
	E(G_m) &\le \inf_{\la\ge0} E(G_{m-1} + \la\varphi_m)
	\\
	&\le E(G_{m-1}) + \inf_{\la\ge0} \Big( -\la t_m A^{-1} (E(G_{m-1}) - E(f^\e)) + 2\rho(E,D,\la) \Big),
\end{align*}
which proves the lemma.
\end{proof}

\begin{proof}[Proof of Theorem~\ref{thm:wbga_conv}]
The error reduction property~\ref{wbga_er} of the class $\mathcal{WBGA}$(co) implies that the sequence of minimizers $\{G_m\}_{m=0}^\infty$ is in $D$ and the sequence $\{E(G_m)\}_{m=0}^\infty$ is non-increasing.
Therefore, we have
\[
	\lim_{m\to \infty} E(G_m) = a \ge \inf_{x\in D}E(x).
\]
Denote
\[
	b := \inf_{x\in D} E(x)
	\ \ \text{and}\ \ 
	\alpha := a - b.
\]
We prove that $\alpha = 0$ by contradiction.
Indeed, assume that $\alpha > 0$.
Then for any $m \ge 0$ we have
\[
	E(G_m) - b \ge \alpha.
\]
We set $\e = \alpha/2$ and find $f^\e \in D$ such that
\[
	E(f^\e) \le b + \e \ \ \text{and}\ \ f^\e/A \in \A_1(\D)
\]
with some $A := A(\e) \ge 1$.
Then by Lemma~\ref{lem:erl} we get
\[
	E(G_m) - E(f^\e) \le E(G_{m-1}) - E(f^\e) + \inf_{\la\ge0} (-\la t_m A^{-1}\alpha/2 + 2\rho(E,D,\la)).
\]
Specify $\theta := \min\left\{ \theta_0,\frac{\alpha}{8A} \right\}$ and take $\la = \xi_m(\rho,\tau,\theta)$ given by~\eqref{eq:theta}.
Then we obtain
\[
	E(G_m) \le E(G_{m-1}) - 2\theta t_m\xi_m.
\]
The assumption
\[
	\sum_{m=1}^\infty t_m\xi_m =\infty
\]
implies a contradiction, which proves the theorem.
\end{proof}

\begin{proof}[Proof of Theorem~\ref{thm:wbga_rate}]
Denote
\[
	a_n := E(G_n) - E(f^\e),
\]
then the sequence $\{a_n\}_{n=0}^\infty$ is non-increasing.
If for some $n \le m$ we have $a_n \le 0$ then $E(G_m) - E(f^\e) \le 0$, which implies
\[
	E(G_m) - \inf_{x \in D} E(x) \le \e,
\]
and hence the statement of the theorem holds.
Thus we assume that $a_n > 0$ for $n \le m$.
By Lemma~\ref{lem:erl} we have
\begin{equation}\label{B3}
	a_m \le a_{m-1} + \inf_{\la\ge0} \left(-\frac{\la t_m a_{m-1}}{B} + 2\gamma \la^q\right).
\end{equation}
Choose $\la$ from the equation
\[
	\frac{\la t_m a_{m-1}}{A} = 4\gamma \la^q,
\]
which implies that
\[
	\la = \left(\frac{ t_m a_{m-1}}{4\gamma A}\right)^{\frac{1}{q-1}} .
\]
Let
\[
	A_q := 2(4\gamma)^{\frac{1}{q-1}}.
\]
Using the notation $p := q/(q-1)$ we get from~\eqref{B3}
\[
	a_m \le a_{m-1}\left(1-\frac{\la t_m}{2A} \right)
	= a_{m-1}\left(1 - \frac{t_m^p}{A_q A^{p}} a_{m-1}^{\frac{1}{q-1}}\right).
\]
Raising both sides of this inequality to the power $1/(q-1)$ and taking into account the inequality $x^r\le x$ for $r\ge 1$, $0\le x\le 1$, we obtain
\[
	a_m^{\frac{1}{q-1}} \le a_{m-1}^{\frac{1}{q-1}} \left(1 - \frac{t^p_m}{A_q A^{p}} a_{m-1}^{\frac{1}{q-1}}\right).
\]
Then Lemma~\ref{lem:y_k} with $y_k := a_k^{\frac{1}{q-1}}$, $n=0$, $w_k=t^p_m/(A_qA^{p})$, which provides
\[
	a_m^{\frac{1}{q-1}} \le C(q,\gamma) A^{p}\left(C(E,q,\gamma) + \sum_{k=1}^m t_k^p\right)^{-1},
\]
that implies
\[
	a_m \le C(q,\gamma) A^q\left(C(E,q,\gamma) + \sum_{k=1}^m t_k^p\right)^{1-q},
\]
which proves the theorem.
\end{proof}

\section{Proofs for Section~\ref{sec:awbga}}\label{sec:proofs_awbga}
In this section we state the proofs for the results from Section~\ref{sec:awbga}.
We begin with the proof that the algorithms stated in Section~\ref{sec:awbga_ga} belong to the class $\mathcal{WBGA}(\Delta,\text{co})$.
\begin{proof}[Proof of Proposition~\ref{prp:ga_awbga}]
It is easy to see that conditions~\ref{awbga_gs} and~\ref{awbga_er} from the definition of the class $\mathcal{WBGA}(\Delta,\text{co})$ are satisfied for all three algorithms.
Condition~\ref{awbga_bd} holds with $C_0 = 1$ since for all three algorithms we have for any $m \ge 1$
\[
	E(G_m) \le E(0) + \de_m \le E(0) + 1.
\]
To guarantee condition~\ref{awbga_bo}, first note that for any $m \ge 1$ and any $u > 0$ the definition of modulus of smoothness~\eqref{eq:mod_smt} provides
\[
	E((1+u) G_m) + E((1-u) G_m) \le 2E(G_m) + 2\rho(E,D_1,u\|G_m\|).
\]
Assume that $\<E'(G_m), G_m\> \ge 0$ (the case $\<E'(G_m), G_m\> < 0$ is handled similarly).
Then from convexity~\eqref{eq:E'_conv1} we get
\[
	E((1+u) G_m) \ge E(G_m) + u \<E'(G_m), G_m\>
\]
and from the definitions of the corresponding algorithms we obtain
\[
	E((1-u) G_m) \ge E(G_m) - \de_m.
\]
Combining the above estimates we deduce
\[
	\<E'(G_m), G_m\> \le \frac{\de_m + 2\rho(E,D_1, u \|G_m\|)}{u}.
\]
Taking infimum over $u > 0$ completes the proof.
\end{proof}

Next, we state necessary technical lemmas that will be utilized in the proof of main results.
\begin{Lemma}[{\cite[Lemma~3.2]{VT148}}]\label{lem:a_delta_0}
Let $\rho(u)$ be a non-negative convex on $[0,1]$ function with the property $\rho(u)/u\to0$ as $u\to 0$.
Assume that a nonnegative sequence $\{\alpha_k\}_{k=1}^\infty$ is such that $\alpha_k\to0$ as $k\to\infty$.
Suppose that a nonnegative sequence $\{a_k\}_{k=0}^\infty$ satisfies the inequalities
\[
	a_m \le a_{m-1} + \inf_{0\le\la\le1}(-\la va_{m-1} + B\rho(\la)) + \alpha_m, \ \ m = 1,2,3,\dots
\]
with positive numbers $v$ and $B$.
Then
\[
	\lim_{m\to\infty} a_m = 0.
\]
\end{Lemma}

\begin{Lemma}[{\cite[Lemma~3.3]{VT148}}]\label{lem:a_delta_q}
Suppose a nonnegative sequence $a_0,a_1,\dots$ satisfies the inequalities for $m = 1,2,3,\dots$
\[
	a_m\le a_{m-1} + \inf_{0\le \la\le 1}(-\la va_{m-1}+B\la^q) + \alpha_m, \ \ \alpha_m \le cm^{-q},
\]
where $q\in (1,2]$, $v\in(0,1]$, and $B > 0$.
Then
\[
	a_m \le C(q,v,B,a_0,c) \, m^{1-q} \le C'(q,B,a_0,c) \, v^{-q} \, m^{1-q}.
\]
\end{Lemma}

Lastly, we establish a generalized version of Lemma~\ref{lem:erl}.
\begin{Lemma}[{{\bf General Error Reduction Lemma}}]\label{lem:gerl}
Let $E$ be a uniformly smooth on $S \subset X$ convex function with the modulus of smoothness $\rho(E,S,u)$.
Take a number $\e \ge 0$ and an element $f^\e \in S$ such that
\[
	E(f^\e) \le \inf_{x\in X} E(x) + \e, \ \ 
	f^\e/B \in \A_1(\D),
\]
with some number $B \ge 1$.
Suppose that $G \in S$ and $\ff \in \D$ satisfy the following conditions
\begin{gather}
	\label{C1}
	\<-E'(G),\ff\> \ge \theta \sup_{g\in \D} \<-E'(G),g\>, \ \ \theta \in (0,1];
	\\
	\label{C2}
	|\<E'(G),G\>| \le \de, \ \ \de \in [0,1].
\end{gather}
Then we have
\begin{align*}
	\inf_{0\le\la\le1} E(G+\la\ff)
	&\le E(G)
	\\
	&+ \inf_{0\le\la\le1} (-\la \theta B^{-1} (E(G) - E(f^\e)) + 2\rho(E,S,\la)) + \de.
\end{align*}
\end{Lemma}
\begin{proof}
If $E(G_{m-1}) - E(f^\e) \le 0$ then the claim of Lemma~\ref{lem:gerl} is trivial.
Assuming $E(G_{m-1}) - E(f^\e) > 0$, Lemma~\ref{lem:E'_rho} provides for any $\la \ge 0$
\[
	E(G + \la \varphi) \le E(G) - \la \<-E'(G),\varphi\> + 2 \rho(E,S,\la).
\]
By~\eqref{C1} and Lemma~\ref{lem:F_A1(D)} we get
\begin{align*}
	\<-E'(G),\varphi\>
	&\ge \theta \sup_{g \in \D} \<-E'(G),g\>
	\\
	&= \theta\sup_{\phi \in \A_1(\D)} \<-E'(G),\phi\> \ge \theta B^{-1} \<-E'(G),f^\e\>.
\end{align*}
By~\eqref{C2} and by convexity~\eqref{eq:E'_conv2} we obtain
\[
	\<-E'(G),f^\e\> = \<-E'(G),f^\e-G\> + \<-E'(G),G\> \ge E(G)-E(f^\e)-\de.
\]
Thus
\[
	E(G + \la\ff) \le E(G) - \la \theta B^{-1} (E(G) - E(f^\e)) + 2\rho(E,S,\la)) + \de,
\]
which proves the lemma.
\end{proof}

\begin{proof}[Proof of Theorem~\ref{thm:awbga_conv}]
Assumption~\ref{awbga_bd} from the definition of the class $\mathcal{WBGA}(\Delta,\text{co})$ implies that for any $m \ge 0$
\[
	E(G_m) \le E(0) + C_0
	\ \ \text{and}\ \
	G_m \in D_1.
\]
Then from Lemma~\ref{lem:gerl} with $S = D_1$, $G = G_{m-1}$, $\ff = \ff_m$, $\de = \e_m$, $\theta = t$, $B = A(\e)$ and property~\ref{awbga_er} from the definition of the class $\mathcal{WBGA}(\Delta,\text{co})$ we obtain
\begin{align}\nonumber
	E(G_m)
	&\le \inf_{0\le\la\le1} E(G_{m-1} + \la\varphi_m) + \de_m
	\\\nonumber
	&\le E(G_{m-1}) + \inf_{0\le\la\le 1} \big(-\la t A^{-1} (E(G_{m-1}) - E(f^\e))
	\\\label{eq:E(G_m)}
	&\phantom{E(G_{m-1}) + \inf_{0\le\la\le 1} \big(-\la}
	+ 2\rho(E,D_1,\la) \big) + \de_m + \e_m.
\end{align}
Denote
\[
	a_n := \max\big\{ E(G_n) - E(f^\e), 0 \big\}.
\]
Note that under our assumptions $t \in (0,1]$ and $A := A(\e) \ge 1$ we always have
\[
	a_{m-1} + \inf_{0\le\la\le 1}(-\la t A^{-1} a_{m-1} + 2\rho(E,D_1,\la)) \ge 0.
\]
Therefore estimate~\eqref{eq:E(G_m)} implies
\begin{equation}\label{eq:a_m}
	a_m \le a_{m-1} + \inf_{0\le\la\le 1} (-\la t A^{-1} a_{m-1} + 2\rho(E,D_1,\la)) + \de_m + \e_m.
\end{equation}
We apply Lemma~\ref{lem:a_delta_0} with $v = tA^{-1}$, $B = 2$, $\rho(u) = \rho(E,D_1,\la)$, and $\alpha_m = \de_m + \e_m$ to obtain
\[
	\lim_{m\to\infty} a_m = 0,
\]
which implies
\[
	\limsup_{m\to\infty} E(G_m) \le \e + \inf_{x\in D_1} E(x)
\]
and, due to the arbitrary nature of choice of $\e > 0$,
\[
	\lim_{m\to\infty} E(G_m) = \inf_{x\in D_1} E(x),
\]
which completes the proof.
\end{proof}

\begin{proof}[Proof of Theorem~\ref{thm:awbga_rate}]
From estimate~\eqref{eq:a_m} we get
\begin{align*}
	a_m
	&\le a_{m-1} + \inf_{0\le\la\le 1} (-\la t A^{-1} a_{m-1} + 2\rho(E,D_1,\la)) + \de_m + \e_m
	\\
	&\le a_{m-1} + \inf_{0\le\la\le 1}(-\la t A^{-1} a_{m-1} + 2\gamma\la^q) + \de_m + \e_m.
\end{align*}
Applying Lemma~\ref{lem:a_delta_q} with $v = t A^{-1}$, $B = 2\gamma$, and $\alpha_m = \de_m + \e_m$ completes the proof.
\end{proof}

\section*{Acknowledgments}
The first author acknowledges support given by the Oak Ridge National Laboratory, which is operated by UT-Battelle, LLC., for the U.S. Department of Energy under Contract DE-AC05-00OR22725.

The work was supported by the Russian Federation Government Grant N{\textsuperscript{\underline{o}}}14.W03.31.0031.

\section*{References}
\bibliographystyle{amsplain}

\end{document}